\documentclass[10pt]{amsart}
\usepackage{amsthm}
\usepackage{tikz-cd} %映射交换图
\usetikzlibrary{cd}
\usepackage{amsmath,amssymb} % For \mathbb command
\usepackage{graphicx} %插入图片的宏包
\usepackage{geometry}
\usepackage[colorlinks,linkcolor=red,anchorcolor=blue,citecolor=green]{hyperref}
\newtheorem{theorem}{Theorem}[section]
\newtheorem{corollary}[theorem]{Corollary}

\newtheorem{proposition}[theorem]{Proposition}
\theoremstyle{definition}
\newtheorem{definition}[theorem]{Definition}
\newtheorem{remark}[theorem]{Remark}
\newtheorem{example}[theorem]{Example}
\newtheorem{question}[theorem]{Question}
\numberwithin{equation}{section}

\usepackage{indentfirst}
\title{Fundamental Generalized Legendrian rack and Classical invariants}
\author{Zhiyun Cheng}
\address{School of Mathematical Sciences, Laboratory of Mathematics and Complex Systems, MOE, Beijing Normal University, Beijing 100875, China}
\email{czy@bnu.edu.cn}
\author{Zhiyi He}
\address{School of Mathematical Sciences, Beijing Normal University, Beijing 100875, China}
\email{202321130080@mail.bnu.edu.cn}
\begin{document}
\begin{abstract}
In this paper, we prove that if two Legendrian knots have isomorphic fundamental GL-racks, then either they have the same Thurston-Bennequin number and the same rotation number, or they have the opposite Thurston-Bennequin numbers and opposite rotation numbers.
\end{abstract}

\maketitle

\section{Introduction}

Quandles and racks are non-associative algebraic structures related to knots and links subject to several axioms derived from Reidemeister moves. In 1982, the notion of quandle was independently introduced by Joyce \cite{Joyce-1982} and Matveev \cite{Matveev-1984}. For a given knot, similar to the knot group, one can define the concept of knot quandle, which was proved to be an almost complete knot invariant. As an extension of quandle, the notion of rack was introduced by Fenn and Rourke \cite{Fenn-Rourke-1992} in 1992 to study framed knots in 3-manifolds. Many knot invariants can be derived from quandles. For example, by counting the homomorphisms from the knot quandle $Q_K$ to a finite quandle $X$, one obtains the associated coloring invariant Col$_X(K)$. A (co)homology theory for quandle was proposed by Carter et al in \cite{state-sum-cocycle-invariant}, and 2-cocycles and 3-cocycles can be used to define state-sum invariants for knots and knotted surfaces, respectively. During the past thirty years, a close connection has been established between quandle theory and Hopf algebras \cite{From_rack_to_pointed_Hopf_algebras}, mapping class group \cite{On_relations_and_homology_of_the_Dehn_quandle}, Yang-Baxter equation \cite{YBE-FA}, and topological spaces \cite{RUBINSZTEIN-RYSZARD}. The reader is referred to \cite{survey-developments,book} for more details and recent developments.

Recently, quandle theory is used to study Legendrian knots by equipped with some extra data. In 2017, Kulkarni and Prathamesh \cite{On-Rack-Invariants-2017} first introduced rack invariants of Legendrian knots. They introduced a family of rack invariants and used them to distinguish the Legendrian trivial knots. In 2021, Ceniceros, Elhamdadi and Nelson \cite{Legendrian-Rack-2021} introduced the Legendrian rack motivated by the front-projection Reidemeister moves for Legendrian knots. The key idea is to associate an automorphism $f$ to each cusp. The resulting counting invariant of Legendrian racks can be used to distinguish some Legendrian knots from their stabilization. Later, the notion of generalized Legendrian rack (or GL-rack for the sake of simplicity), which is an enhanced version of Legendrian rack, was introduced by Kimura \cite{Bi-Legendrian-2023} and Karmakar-Starf-Singh \cite{karmakar2024generalisedlegendrianrackslegendrian} independently. The key to this concept lies in distinguishing between the up cusp and the down cusp. The corresponding invariant of GL-racks can distinguish infinitely many oriented Legendrian unknots and trefoil knots.

In \cite{Bi-Legendrian-2023}, Kimura proved that if two Legendrian knots have the same classical invariants, i.e. the same topological knot type, the same Thurston-Bennequin number and the same rotation number, then for any finite GL-quandle, they share the same coloring number. The main result of this paper is to present the following relationship between the fundamental GL-rack and classical invariants of Legendrian knots, which can be regarded as a partial converse of Kimura's result.

\begin{theorem}\label{theorem1}
Let $K_1, K_2$ be two Legendrian knots with isomorphic fundamental GL-racks, then either $tb(K_1)=tb(K_2)$ and $rot(K_1)=rot(K_2)$, or $tb(K_1)=-tb(K_2)$ and $rot(K_1)=-rot(K_2)$.
\end{theorem}

The remainder of this paper is arranged as follows. In Section \ref{section2} we take a quick review of Legendrian knot theory, including the classical invariants, GL-rack, and fundamental GL-rack of Legendrian knots. Section 3 is devoted to give the proof of Theorem \ref{theorem1}.

\section{Legendrian knots and Fundamental GL-rack}\label{section2}
\subsection{Legendrian knots}
In this section, we review some basics of Legendrian knots. A \textit{contact structure} on an oriented 3-manifold $M$ is a completely non-integrable plane field $\xi$ in the tangent bundle of $M$. A 1-form $\alpha$ on $M$ is called a \textit{contact form} if $\alpha\wedge d\alpha$ is a volume form. A manifold $M$ equipped with a contact structure $\xi$ is called a \textit{contact manifold}, denoted by $(M,\xi)$. For a given contact form $\alpha$, the kernel of it defines a contact structure $\xi=\ker\alpha$. Note that two contact forms $\alpha$ and $f\alpha$ give rise to the same contact structure, provided that $f$ is a smooth non-vanishing function on $M$.

\begin{example}
Let $(x,y,z)$ be the standard coordinates on $\mathbb{R}^3$, consider the contact form $\alpha=dz-ydx$ and the corresponding contact structure $\xi_{std} =\ker\alpha=\text{span}\{\frac{\partial}{\partial y}, \frac{\partial}{\partial x}+y\frac{\partial}{\partial z}\}$. Then $\xi_{std}$ is a contact structure on $\mathbb{R}^3$, called the \textit{standard contact structure} on $\mathbb{R}^3$.
\end{example}

According to Darboux’s Theorem, any point in a contact manifold $M$ has an open neighborhood that is diffeomorphic to an open neighborhood of the origin in $\mathbb{R}^3$ by a diffeomorphism that takes the contact structure on $M$ to the standard contact structure $\xi_{std}$ on $\mathbb{R}^3$. From now on, we focus our attention on the standard contact 3-manifold $(\mathbb{R}^3, \xi_{std})$.

\begin{definition}
A \textit{Legendrian knot} $K$ in $(\mathbb{R}^3, \xi_{std})$ is a smoothly embedded $S^1$ that is always tangent to $\xi$, i.e. $T_xK\in\xi_x, \forall x\in K$.
\end{definition}

Let $\phi(\theta) = (x(\theta),y(\theta),z(\theta))$ be a parameterization of a Legendrian knot in $(\mathbb{R}^3,\xi_{std})$. Then $\phi'(\theta) = (x'(\theta),y'(\theta),z'(\theta)) \in \xi_{\phi(\theta)} $, it follows that $z'(\theta)-yx'(\theta)=0$. There are two projections of Legendrian knots: the \textit{front projection} \[\pi: \mathbb{R}^3 \rightarrow \mathbb{R}^2: (x,y,z) \mapsto (x,z),\] and the \textit{Lagrangian projection} \[\pi: \mathbb{R}^3 \rightarrow \mathbb{R}^2: (x,y,z) \mapsto (x,y).\]
Consider the front projection, since $y=\frac{z'(\theta)}{x'(\theta)}$, we obtain the following two facts:
\begin{enumerate}
\item front projection has no vertical tangencies;
\item at each crossing point, the slope of the overcrossing is more negative than the undercrossing, see Figure \ref{fig2}.
\end{enumerate}

\tikzset{every picture/.style={line width=0.75pt}} %set default line width to 0.75pt        
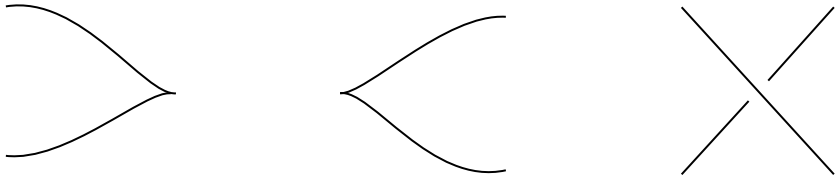
\begin{figure}[h]
\begin{tikzpicture}[x=0.75pt,y=0.75pt,yscale=-1,xscale=1]

%Curve Lines [id:da9008091094960422] 
\draw    (236,79.53) .. controls (273.59,73.27) and (305.21,124.62) .. (320.89,123.25) ;
%Curve Lines [id:da5270246655885171] 
\draw    (236,154.68) .. controls (266.04,158.09) and (305.55,121.76) .. (318.27,123.25) ;
%Curve Lines [id:da04954669110387189] 
\draw    (485.57,161.95) .. controls (448.33,170.2) and (418.42,120.99) .. (402.83,123.18) ;
%Curve Lines [id:da39652410750170575] 
\draw    (485.57,84.75) .. controls (455.4,82.93) and (415.61,124) .. (402.83,123.18) ;
%Straight Lines [id:da00182658577926742] 
\draw    (573.33,79.94) -- (649.33,163.97) ;
%Straight Lines [id:da7626482893292794] 
\draw    (649.33,79.94) -- (616.55,116.86) ;
%Straight Lines [id:da9189110534833317] 
\draw    (606.76,127.11) -- (573.33,164) ;
\end{tikzpicture}
\caption{cusps and crossing}
\label{fig2}
\end{figure}

\begin{definition}
Two Legendrian knots $K_1, K_2$ are called \textit{Legendrian equivalent} if there exists an isotopy through Legendrian knots in $(\mathbb{R}^3,\xi_{std})$ from $K_1$ to $K_2$.
\end{definition}

Obviously, if two Legendrian knots are Legendrian equivalent, they must be of the same topological knot type. The three local moves of front diagrams shown in Figure \ref{fig3} are called \textit{Legendrian Reidemeister moves}. It is well known that two front diagrams represent the same Legendrian knot if and only if they are related by a sequence of Legendrian Reidemeister moves.

\tikzset{every picture/.style={line width=0.75pt}} %set default line width to 0.75pt        
\begin{figure}[h]
\begin{tikzpicture}[x=0.75pt,y=0.75pt,yscale=-1,xscale=1]
%uncomment if require: \path (0,321); %set diagram left start at 0, and has height of 321

%Curve Lines [id:da919621188040473] 
\draw    (316.9,70.73) .. controls (369.06,48.61) and (367.99,48.61) .. (421.74,70.27) ;
%Curve Lines [id:da1016899669121264] 
\draw    (495.19,58.75) .. controls (504.77,60.13) and (520.2,46.3) .. (525.52,42.61) .. controls (530.85,38.93) and (549.47,31.09) .. (564.91,44) .. controls (580.34,56.9) and (593.11,60.59) .. (599.5,60.59) ;
%Curve Lines [id:da17518506838772496] 
\draw    (495.19,58.75) .. controls (510.09,61.51) and (566.15,71.41) .. (598.52,102.54) ;
%Curve Lines [id:da43458012936759527] 
\draw    (552.13,70.27) .. controls (565.97,61.51) and (591.52,62.44) .. (599.5,60.59) ;
%Curve Lines [id:da5197682113356401] 
\draw    (499.45,102.54) .. controls (506.9,95.62) and (523.39,79.95) .. (540.96,75.34) ;
%Straight Lines [id:da29592223537992957] 
\draw    (457.91,67.87) -- (485.16,67.96) ;
\draw [shift={(487.16,67.97)}, rotate = 180.19] [color={rgb, 255:red, 0; green, 0; blue, 0 }  ][line width=0.75]    (10.93,-3.29) .. controls (6.95,-1.4) and (3.31,-0.3) .. (0,0) .. controls (3.31,0.3) and (6.95,1.4) .. (10.93,3.29)   ;
%Straight Lines [id:da550949485865951] 
\draw    (457.91,67.87) -- (434.14,68.42) ;
\draw [shift={(432.14,68.47)}, rotate = 358.67] [color={rgb, 255:red, 0; green, 0; blue, 0 }  ][line width=0.75]    (10.93,-3.29) .. controls (6.95,-1.4) and (3.31,-0.3) .. (0,0) .. controls (3.31,0.3) and (6.95,1.4) .. (10.93,3.29)   ;
%Curve Lines [id:da5946785653762862] 
\draw    (316.86,123.36) .. controls (349.58,119.64) and (377.11,150.09) .. (390.76,149.28) ;
%Curve Lines [id:da5504372836415122] 
\draw    (316.86,167.92) .. controls (343.01,169.94) and (377.41,148.4) .. (388.48,149.28) ;
%Straight Lines [id:da5887475621085283] 
\draw    (396.36,120.98) -- (428.62,177.37) ;
%Curve Lines [id:da8773683285296897] 
\draw    (546.44,134.02) .. controls (551.6,136.19) and (583.8,154.71) .. (594.14,154.17) ;
%Curve Lines [id:da8596553402921397] 
\draw    (564.57,161.1) .. controls (568.93,159.6) and (585.95,153.8) .. (591.29,154.17) ;
%Straight Lines [id:da7823157258328821] 
\draw    (529.26,117.75) -- (569.71,179.67) ;
%Curve Lines [id:da39863651355432417] 
\draw    (500.1,127.84) .. controls (510.13,126.84) and (519.77,128.08) .. (528.87,130.52) ;
%Curve Lines [id:da7209594356104824] 
\draw    (499.41,174.1) .. controls (514.92,175.15) and (532.74,170.41) .. (548.69,165.23) ;
%Straight Lines [id:da8474323779464444] 
\draw    (455.75,147.62) -- (482.98,148.06) ;
\draw [shift={(484.98,148.09)}, rotate = 180.93] [color={rgb, 255:red, 0; green, 0; blue, 0 }  ][line width=0.75]    (10.93,-3.29) .. controls (6.95,-1.4) and (3.31,-0.3) .. (0,0) .. controls (3.31,0.3) and (6.95,1.4) .. (10.93,3.29)   ;
%Straight Lines [id:da12282854274980681] 
\draw    (455.75,147.62) -- (431.98,148.17) ;
\draw [shift={(429.98,148.22)}, rotate = 358.67] [color={rgb, 255:red, 0; green, 0; blue, 0 }  ][line width=0.75]    (10.93,-3.29) .. controls (6.95,-1.4) and (3.31,-0.3) .. (0,0) .. controls (3.31,0.3) and (6.95,1.4) .. (10.93,3.29)   ;
%Straight Lines [id:da25123364097156453] 
\draw    (316.9,199.43) -- (411.66,276.62) ;
%Straight Lines [id:da6250462528009958] 
\draw    (389.64,217.7) -- (370.79,233.35) ;
%Straight Lines [id:da23053250989962681] 
\draw    (358.58,242.76) -- (316.9,276.65) ;
%Curve Lines [id:da45891296047959496] 
\draw    (342.81,213.65) .. controls (345.56,211.73) and (367.21,185.28) .. (412.59,232.3) ;
%Straight Lines [id:da29930627479925886] 
\draw    (411.59,199.47) -- (395.89,212.51) ;
%Curve Lines [id:da4295084867924257] 
\draw    (320.4,228.56) .. controls (326.36,224.83) and (331.41,221.51) .. (335.75,218.57) ;
%Straight Lines [id:da4037261336081496] 
\draw    (538.11,242.78) -- (517.5,259.53) ;
%Straight Lines [id:da9043210657242517] 
\draw    (495.87,199.44) -- (590.63,276.64) ;
%Straight Lines [id:da04990250436301902] 
\draw    (568.61,217.72) -- (549.76,233.36) ;
%Straight Lines [id:da799815434742848] 
\draw    (509.04,265.49) -- (495.3,276.67) ;
%Straight Lines [id:da28079098831101423] 
\draw    (590.56,199.49) -- (568.61,217.72) ;
%Shape: Boxed Bezier Curve [id:dp8095014412436277] 
\draw    (562.48,262.29) .. controls (559.74,264.21) and (538.18,290.7) .. (492.64,243.79) ;
%Curve Lines [id:da5931393052690758] 
\draw    (572.24,254.58) .. controls (578,251.29) and (583.05,245.76) .. (585.9,240.42) ;
%Straight Lines [id:da9352323671478866] 
\draw    (448.92,227.33) -- (477.74,227.78) ;
\draw [shift={(479.74,227.81)}, rotate = 180.88] [color={rgb, 255:red, 0; green, 0; blue, 0 }  ][line width=0.75]    (10.93,-3.29) .. controls (6.95,-1.4) and (3.31,-0.3) .. (0,0) .. controls (3.31,0.3) and (6.95,1.4) .. (10.93,3.29)   ;
%Straight Lines [id:da44265199169786584] 
\draw    (448.92,227.33) -- (423.74,227.89) ;
\draw [shift={(421.74,227.93)}, rotate = 358.74] [color={rgb, 255:red, 0; green, 0; blue, 0 }  ][line width=0.75]    (10.93,-3.29) .. controls (6.95,-1.4) and (3.31,-0.3) .. (0,0) .. controls (3.31,0.3) and (6.95,1.4) .. (10.93,3.29)   ;
\end{tikzpicture}
\caption{Legendrian Reidemeister moves}
\label{fig3}
\end{figure}
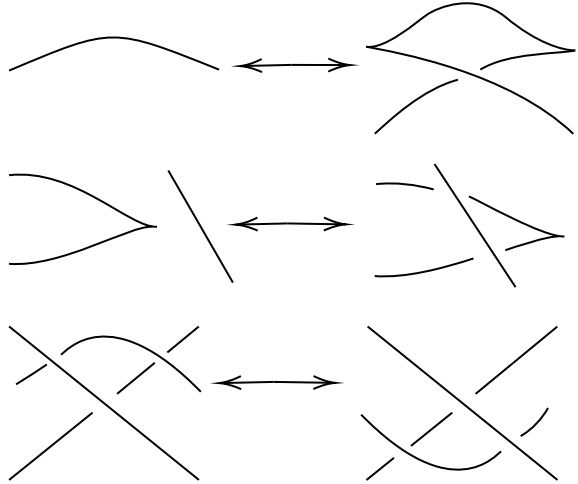

In order to distinguish different Legendrian knots, many invariants have been studied in the literature. Here we briefly review the classical invariants of Legendrian knots. The reader is referred to \cite{Handbook-Legendrian-transversal} for more details and recent progress.

\begin{itemize}
\item \textbf{Topological knot type}. Obviously, the underlying topological knot type of a Legendrian knot $K$ is an invariant, which is usually denoted by $k(K)$. For the set of all the Legendrian knots with the same underlying topological knot type $\mathcal{K}$, let us use $\mathcal{L}(\mathcal{K})$ to denote it.
\item \textbf{Thurston-Bennequin number}. For a Legendrian knot $K$, the \textit{Thurston-Bennequin number} $tb(K)$ measures the number of twistings of $\xi$ around $K$. It can be read from the crossings and cusps of its oriented front projection $D(K)$ as follows
\begin{center}
$tb(K)=w(D(K))-\frac{1}{2} (\text{number of cusps})$,
\end{center}
where $w(D(K))$ denotes the writhe of $D(K)$.
\item \textbf{Rotation number}. For a Legendrian knot $K$, the \textit{rotation number} $rot(K)$ can be regarded as the winding number of a nonzero tangent vector of $K$ in the fiber of a trivialization $\xi|_K=K\times\mathbb{R}^2$, induced by a Seifert surface of $K$. It also can be directly read from its oriented front diagram $D(K)$ as follows
\begin{center}
$rot(K)=\frac{1}{2}(D-U)$,
\end{center}
where $D$ and $U$ denote the number of down cusps and up cusps, respectively.
\end{itemize}

The pair of the Thurston-Bennequin number and the rotation number $\{tb(K), rot(K)\}$ completely classifies the Legendrian isotopy classes of some knot types $\mathcal{K}$, such as the unknot \cite{MR1619122,Eliashberg-Yakov-Fraser-Maia}, the figure eight knot and torus knots \cite{8-tours-knot}.

Given a Legendrian knot $K$, one can obtain some other Legendrian knots with the same topological knot type by stabilization $S_{\pm}$, see Figure \ref{fig4}. It is easy to find that 
\begin{center}
$tb(S_{\pm}(K))=tb(K)-1$
\end{center}
and 
\begin{center}
$rot(S_{\pm}(K))=rot(K)\pm1$.
\end{center}
Notice that $S_{\pm}(K)$ does not depend on the place where the stabilization is done. Stabilization plays an important role in Legendrian knot theory, since two Legendrian knots with the same topological knot type are Legendrian isotopic after each has been stabilized several times.

\tikzset{every picture/.style={line width=0.75pt}} %set default line width to 0.75pt        
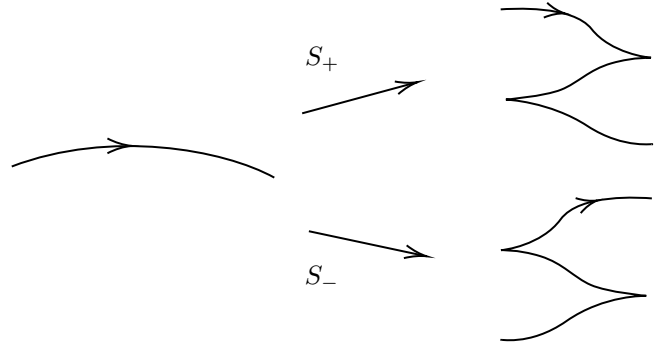
\begin{figure}[h]
\begin{tikzpicture}[x=0.75pt,y=0.75pt,yscale=-1,xscale=1]
%uncomment if require: \path (0,374); %set diagram left start at 0, and has height of 374

%Shape: Boxed Bezier Curve [id:dp5404097881495223] 
\draw    (550.72,85.58) .. controls (556.44,85.21) and (586.58,83.3) .. (596.11,95.47) .. controls (605.65,107.64) and (624.51,109.91) .. (625.78,109.91) .. controls (627.05,109.91) and (617.47,109.37) .. (606.03,112.69) .. controls (594.59,116) and (589.47,122.53) .. (580.92,125.96) .. controls (572.37,129.39) and (553.31,130.93) .. (553.31,130.93) .. controls (553.31,130.93) and (575.16,131.84) .. (593.99,144.21) .. controls (612.83,156.57) and (628.26,152.63) .. (626.14,153.55) ;
\draw   (573.93,82.24) .. controls (576.69,84.86) and (579.65,86.64) .. (582.78,87.61) .. controls (579.46,87.47) and (575.95,88.18) .. (572.25,89.7) ;
%Shape: Boxed Bezier Curve [id:dp14694894641426348] 
\draw    (626.19,180.66) .. controls (620.46,180.24) and (590.27,178.07) .. (580.65,191.19) .. controls (571.04,204.31) and (552.13,206.7) .. (550.86,206.69) .. controls (549.58,206.69) and (559.18,206.13) .. (570.64,209.76) .. controls (582.09,213.39) and (587.18,220.45) .. (595.74,224.2) .. controls (604.3,227.94) and (623.39,229.66) .. (623.39,229.66) .. controls (623.39,229.66) and (602.52,229.04) .. (583.58,242.34) .. controls (564.65,255.64) and (549.1,250.98) .. (551.22,251.98) ;
\draw   (588.14,180.51) .. controls (591.99,181.74) and (595.58,182.14) .. (598.88,181.69) .. controls (595.86,182.98) and (593.12,185.09) .. (590.67,188.06) ;
%Shape: Boxed Line [id:dp05294605211583092] 
\draw    (451.63,137.87) -- (507.35,122.67) ;
\draw [shift={(509.28,122.15)}, rotate = 164.75] [color={rgb, 255:red, 0; green, 0; blue, 0 }  ][line width=0.75]    (10.93,-3.29) .. controls (6.95,-1.4) and (3.31,-0.3) .. (0,0) .. controls (3.31,0.3) and (6.95,1.4) .. (10.93,3.29)   ;
%Shape: Boxed Line [id:dp07650895907600197] 
\draw    (454.98,197.16) -- (511.82,209.36) ;
\draw [shift={(513.78,209.78)}, rotate = 192.11] [color={rgb, 255:red, 0; green, 0; blue, 0 }  ][line width=0.75]    (10.93,-3.29) .. controls (6.95,-1.4) and (3.31,-0.3) .. (0,0) .. controls (3.31,0.3) and (6.95,1.4) .. (10.93,3.29)   ;
%Curve Lines [id:da3851152715961178] 
\draw    (306.67,164.61) .. controls (355.95,145.71) and (412.42,156.09) .. (437.72,170.2) ;
\draw   (354.42,150.5) .. controls (358.42,152.57) and (362.42,153.82) .. (366.41,154.23) .. controls (362.42,154.64) and (358.42,155.89) .. (354.42,157.96) ;

% Text Node
\draw (451.7,102.91) node [anchor=north west][inner sep=0.75pt]   [align=left] {$\displaystyle S_{+}$};
% Text Node
\draw (451.44,213.1) node [anchor=north west][inner sep=0.75pt]   [align=left] {$\displaystyle S_{-}$};
\end{tikzpicture}
\caption{Stabilization}
\label{fig4}
\end{figure}

\subsection{Fundamental GL-rack}
Recall that a \textit{rack} is a pair $(X, \ast)$ where $X$ is a set and $\ast$ is a binary operation satisfying the following axioms:
\begin{enumerate}
\item the map $s_x: X\rightarrow X$, where $s_x(y)=y\ast x$ is a bijection for all $x\in X$;
\item $(x\ast y)\ast z=(x\ast z)\ast(y\ast z)$ for all $x, y, z\in X$.
\end{enumerate}
A \textit{quandle} is a rack satisfying that $x\ast x = x$ for all $x\in X$. These three axioms correspond to the Reidemeister moves in knot theory.

\begin{example}
Here we list some examples of rack:
\begin{enumerate}
\item Any set $X$ can be regarded as a quandle if we set $x\ast y=x$ for all $x, y\in X$.
\item Let $G$ be a group, we set $x\ast_1 y=y^{-1}xy$ and $x\ast_2 y=yx^{-1}y$, then $(G, \ast_1)$ and $(G, \ast_2)$ are called the \textit{conjugation quandle} and \textit{core quandle} associated to $G$, respectively.
\item Let $X$ be a set, $(X,\ast_{\sigma})$ is called the \textit{permutation rack} if $x \ast_{\sigma} y = \sigma(x)$ for all $x,y \in X$, where $\sigma$ is a fixed permutation on $X$. Clearly, a permutation rack $(X, \ast_{\sigma})$ is a quandle if and only if $\sigma = id_X$.
\end{enumerate}
\end{example}

\begin{definition}\cite{karmakar2024generalisedlegendrianrackslegendrian,Bi-Legendrian-2023}
A \textit{generalized Legendrian rack} (GL-rack) is a quadruple $(X, \ast, u, d)$, where $(X, \ast)$ is a rack and $u, d$ are two maps on $X$ subject to the following conditions:
\begin{enumerate}
\item $ud(x\ast x)=du(x\ast x)=x$;
\item $u(x\ast y)=u(x)\ast y$, $d(x\ast y)=d(x)\ast y$;
\item $x\ast u(y)=x\ast d(y)=x\ast y$.
\end{enumerate}
\end{definition}

Note that if $(X,\ast, u, d)$ is a GL-rack with $u=d$, then it degenerates into the notion of \textit{Legendrian rack} introduced in \cite{Legendrian-Rack-2021}. There are some facts about GL-rack.

\begin{proposition}\cite{karmakar2024generalisedlegendrianrackslegendrian}
Let $(X,\ast, u, d)$ be a GL-rack, then the following hold:
\begin{enumerate}
\item $(X,\ast)$ is a quandle if and only if $ud=du=id_X$;
\item Both the maps $u$ and $d$ are automorphisms of the underlying rack $(X, \ast)$;
\item $ud(x)\ast x=du(x)\ast x=x$ for all $x \in X$. In particular, $ud(x)=du(x), \forall x\in X$;
\item $u(x\ast^{-1} y)=u(x)\ast^{-1}y$ and $d(x\ast^{-1}y)=d(x)\ast^{-1}y$ for all $x, y\in X$.
\end{enumerate}
\end{proposition}

\begin{proof}
To ensure that this article is self-contained, we provide a brief proof here. 
\begin{enumerate}
\item The first proposition follows directly from the definition.
\item According to the second of third conditions of GL-rack, we have 
\begin{center}
$u(x\ast y)=u(x)\ast y=u(x)\ast u(y)$.
\end{center}
Therefore $u$ defines a homomorphisms from $X$ to itself. The surjectivity of $u$ follows from the condition $ud(x\ast x)=x$. For the injectivity of $u$, assume $u(x)=u(y)$, then
\begin{center}
$x=du(x\ast x)=d(u(x)\ast u(x))=d(u(y)\ast u(y))=du(y\ast y)=y$,
\end{center}
which implies that $u$ is injective. This finishes the proof that $u$ is an automorphism of the underlying rack $(X, \ast)$. Similarly, one can show that $d$ is also an automorphism of $(X, \ast)$.
\item One observes that $ud(x)\ast x=u(d(x)\ast x)=u(d(x\ast x))=x$. The equality $du(x)\ast x=x$ can be proved in a comparable approach.
\item One observes that $u(x\ast^{-1}y)\ast y=u((x\ast^{-1}y)\ast y)=u(x)$, which means $u(x\ast^{-1}y)=u(x)\ast^{-1}y$. The other equality can be obtained in a similar way.
\end{enumerate}
\end{proof}

\begin{remark}
It is worth noting that for a given GL-rack $(X, \ast, u, d)$, actually the map $d$ is completely determined by the underlying rack $(X, \ast)$ and the map $u$. Actually, if we denote $\theta(x)=x\ast x$, then according to the definition of rack, the map $\theta: X\to X$ is invertible and $d=\theta^{-1}u^{-1}$ \cite[Proposition 3.11]{Ta-2025}.
\end{remark}

\begin{definition}
Let $(X, \ast_1 , u_1, d_1)$ and $(Y, \ast_2, u_2, d_2 )$ be two GL-racks. Then a \textit{GL-rack homomorphism} $\phi:(X, \ast_1 , u_1, d_1) \rightarrow (Y, \ast_2, u_2, d_2 )$ is a map $\phi: X \rightarrow Y$ such that $\phi(x\ast_1 y)=\phi(x)\ast_2\phi(y)$, $\phi u_1=u_2 \phi$ and $\phi d_1=d_2 \phi$.
\end{definition}

Since $s_x(y)=y\ast x$ is a rack automorphism of $(X, \ast)$, the second condition of GL-rack implies that $s_x$ is a GL-rack automorphism. Then all GL-rack and their morphisms form a category, more algebraic results of it can be found in \cite{Ta-2025}.

\begin{example}
Here we list some examples of GL-racks.
\begin{enumerate}
\item Let $(X, \ast)$ be a quandle, then $(X, \ast ,id_X, id_X)$ is a GL-rack, called the trivial GL-rack;
\item Let $G$ be a group, consider the conjugation quandle $(G, \ast)$ and set $u(x)=zx, d(x)=z^{-1}x$. Here $z$ is a central element of $G$, then $(G, \ast, u, d)$ is a GL-rack.
\item Let $(X, \ast_{\sigma})$ be a permutation rack such that $du=\sigma^{-1}$, then $(X, \ast_{\sigma}, u, d)$ is a GL-rack.
\end{enumerate}
\end{example}

Analogously to the fundamental quandle of topological knots, one can define the fundamental GL-rack of Legendrian knots. In order to introduce the fundamental GL-rack, we need the notion of free GL-rack. 

\begin{definition}\cite{karmakar2024generalisedlegendrianrackslegendrian}.
Let $X$ be a non-empty set, the universe of words generated by $X$ is a set $W(X)$ satisfying the following:
\begin{enumerate}
    \item $x\in W(X)$ for all $x \in X$;
    \item $x\ast y, x\ast^{-1}y, u(x), d(x)\in W(X) $ for all $x, y\in X$.
\end{enumerate}
The \textit{free GL-rack} $FGLR(X)$ is a set of equivalence classes of elements of $W(X)$ modulo the equivalence relation generated by the following relations:
    \begin{itemize}
        \item $(x\ast y) \ast^{-1}y\sim (x\ast^{-1} y) \ast y\sim x $
        \item $(x\ast y)\ast z \sim (x\ast z)\ast (y\ast z)$
        \item $u(d(x\ast x))\sim d(u(x \ast x))\sim x$
        \item $u(x\ast y) \sim u(x) \ast y, d(x\ast y) \sim d(x) \ast y$
        \item $x \ast u(y)\sim x \ast d(y)\sim x\ast y$
    \end{itemize}
\end{definition}

Let $D(K)$ be a front projection of a Legendrian knot $K$ and $X$ be the set of arcs of $D(K)$. Here an arc means a part of $D(K)$ between two undercrossing points, or an undercrossing point and a cusp, or two cusps. Each crossing point and each cusp provide a relation as shown in Figure \ref{fig5}. Then, the free GL-rack $(FGLR(D(K)), \ast, u, d)$ associated to the front diagram $D(K)$ is called the \textit{fundamental GL-rack} of $K$, denoted by $GLR_K$.

\tikzset{every picture/.style={line width=0.75pt}} %set default line width to 0.75pt        
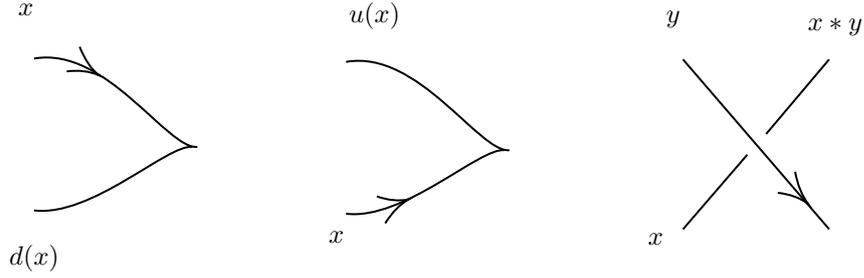
\begin{figure}[h]
\begin{tikzpicture}[x=0.75pt,y=0.75pt,yscale=-1,xscale=1]

%Curve Lines [id:da8972340036502057] 
\draw    (250.52,96.34) .. controls (286.61,89.97) and (316.98,142.18) .. (332.03,140.78) ;
%Curve Lines [id:da332070209027829] 
\draw    (250.52,172.73) .. controls (279.36,176.2) and (317.31,139.27) .. (329.52,140.78) ;
%Straight Lines [id:da29732918584743484] 
\draw    (574.43,96.76) -- (647.41,182.17) ;
%Straight Lines [id:da9135513234673247] 
\draw    (647.41,96.76) -- (615.94,134.29) ;
%Straight Lines [id:da6177078764339164] 
\draw    (606.53,144.71) -- (574.43,182.21) ;
%Curve Lines [id:da8599781598238188] 
\draw    (406.27,98.01) .. controls (442.36,91.64) and (472.73,143.84) .. (487.78,142.45) ;
%Curve Lines [id:da3740141530194604] 
\draw    (406.27,174.4) .. controls (435.11,177.87) and (473.05,140.94) .. (485.27,142.45) ;
\draw   (273.28,90.42) .. controls (275.69,96.55) and (278.81,101.3) .. (282.65,104.67) .. controls (278.09,102.68) and (272.82,102.06) .. (266.82,102.81) ;
\draw   (421.62,165.67) .. controls (427.37,167.68) and (432.66,168.19) .. (437.47,167.19) .. controls (433.12,169.69) and (429.23,173.69) .. (425.81,179.2) ;
\draw   (630.42,153.27) .. controls (631.61,159.82) and (633.77,165.19) .. (636.88,169.38) .. controls (632.8,166.37) and (627.76,164.54) .. (621.76,163.88) ;

% Text Node
\draw (241.13,67.55) node [anchor=north west][inner sep=0.75pt]   [align=left] {$\displaystyle x$};
% Text Node
\draw (396.13,181.72) node [anchor=north west][inner sep=0.75pt]   [align=left] {$\displaystyle x$};
% Text Node
\draw (555.63,182.55) node [anchor=north west][inner sep=0.75pt]   [align=left] {$\displaystyle x$};
% Text Node
\draw (406.13,66.72) node [anchor=north west][inner sep=0.75pt]   [align=left] {$\displaystyle u( x)$};
% Text Node
\draw (236.66,188.38) node [anchor=north west][inner sep=0.75pt]   [align=left] {$\displaystyle d( x)$};
% Text Node
\draw (564.66,70.88) node [anchor=north west][inner sep=0.75pt]   [align=left] {$\displaystyle y$};
% Text Node
\draw (635.3,74.22) node [anchor=north west][inner sep=0.75pt]   [align=left] {$\displaystyle x\ast y$};
\end{tikzpicture}
\caption{Relations corresponding to cusps and crossing point}
\label{fig5}
\end{figure}

It turns out that $GLR_K$ does not depend on the choice of the front projection.

\begin{theorem}\cite{karmakar2024generalisedlegendrianrackslegendrian}
The fundamental GL-rack $GLR_K$ of a front projection $D(K)$ of an oriented Legendrian knot $K$ is preserved under the Legendrian Reidemeister moves.
\end{theorem}

Let $D(K)$ be a front projection of a Legendrian knot $K$, a coloring of $D(K)$ by a fixed finite GL-rack $(X, \ast, u, d)$ is a map from $GLR_K$ to $X$, such that at each crossing and each cusp the relations shown in Figure \ref{fig5} are satisfied. Obviously, there exists a one-to-one correspondence between the set of GL-rack homomorphisms from $GLR_K$ to $X$ and the set of colorings of $D(K)$ by $X$. The number of $(X, \ast, u, d)$ colorings of $D(K)$ gives rise to an integer-valued invariant of $K$. We denote it by Col$_X(K)$.

\subsection{Two examples}
In this subsection, we give two examples of Legendrian knots and their fundamental GL-racks.

\begin{example}\label{exmaple1}
Two Legendrian knots with the same topological knot type $5_1$ are shown in Figure \ref{fig6}, say $K_1$ (left) and $K_2$ (right). It is easy to find that the fundamental GL-rack $GLR_{K_1}$ is generated by $\{ x_1,x_2,x_3,x_4,x_5 \}$ with relations:
\[
\begin{cases}
    ud(x_1) \ast^{-1} x_4 = x_2\\
    d^2(x_2) \ast^{-1} x_5 = x_3\\
    ud(x_3) \ast^{-1} x_1 = x_4\\
    d^2(x_4) \ast^{-1} x_2 = x_5\\
    d^2(x_5) \ast^{-1} x_3 = x_1\\
\end{cases}
\]
and the fundamental GL-rack $GLR_{K_2}$ is generated by $\{ y_1,y_2,y_3,y_4,y_5,y_6,y_7,y_8 \}$ with relations:
\[
\begin{cases}
    u(y_1) \ast^{-1} y_7 = y_2\\
    y_2 \ast^{-1} y_4 = y_3\\
    d(y_3) \ast^{-1} y_7 = y_4\\
    y_4 \ast^{-1} y_1 = y_5\\
    d(y_5) \ast^{-1} y_7 = y_6\\
    y_6 \ast^{-1} y_4 = y_7\\
    d(y_7) \ast^{-1} y_1 = y_8\\
    y_8 \ast^{-1} y_5 = y_1
\end{cases}
\]
Now we consider the permutation GL-rack $(\mathbb{Z}_9, \ast, u, d)$ with $a\ast b=\sigma(a)$, $u=\sigma^{-1}$ and $d=id_{\mathbb{Z}_9}$, where $a, b\in\mathbb{Z}_9$ and $\sigma$ is the $9$-cycle. Choose an element $a\in\mathbb{Z}_9$, define $\phi: GLR_{K_2}\to(\mathbb{Z}_9, \ast, u, d)$ by setting $\phi(y_1)=a$. Then $\phi(y_i)=\sigma^{-i}(a)$ $(2\leq i\leq8)$ and $\phi$ is a GL-rack homomorphism. It follows that Col$_{\mathbb{Z}_9}(K_2)=9$. On the other hand, if there exists a GL-rack homomorphism $\psi: GLR_{K_1}\to(\mathbb{Z}_9, \ast, u, d)$ such that $\psi(x_1)=b$ for some $b\in\mathbb{Z}_9$. Then $\psi(x_1)=\psi(d^2(x_5)\ast^{-1} x_3)=\cdots=\sigma^{-7}(\psi(x_1))$, which is impossible. Therefore, we obtain Col$_{\mathbb{Z}_9}(K_1)=0$. It shows that $GLR_{K_1}$ is not isomorphic to $GLR_{K_2}$.

\tikzset{every picture/.style={line width=0.75pt}} %set default line width to 0.75pt        
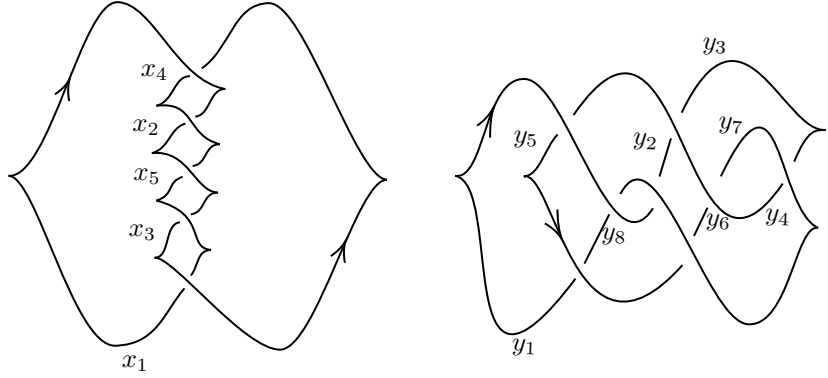
\begin{figure}
\begin{tikzpicture}[x=0.75pt,y=0.75pt,yscale=-1,xscale=1]
%uncomment if require: \path (0,370); %set diagram left start at 0, and has height of 370

%Curve Lines [id:da7920185941541555] 
\draw    (275,175.43) .. controls (295.53,176.91) and (311.52,87.8) .. (329.34,88.07) .. controls (347.16,88.34) and (358.55,127.36) .. (383.22,131.75) ;
%Curve Lines [id:da021887100243184077] 
\draw    (275,175.43) .. controls (288.74,174.93) and (307.37,261.44) .. (328.32,260.87) .. controls (349.26,260.3) and (352.89,243.42) .. (362.85,232.01) ;
%Curve Lines [id:da837917935966189] 
\draw    (463.83,177.41) .. controls (448.13,176.42) and (422.78,88.29) .. (404.96,88.57) .. controls (387.14,88.84) and (386.04,111.13) .. (371.48,121.38) ;
%Curve Lines [id:da4499835521743898] 
\draw    (463.83,177.41) .. controls (449.03,176.91) and (428.4,262.87) .. (410.51,262.85) .. controls (392.61,262.83) and (356.96,215.63) .. (347.9,216.62) ;
%Curve Lines [id:da49336646310760557] 
\draw    (365.56,125.3) .. controls (357.41,125.3) and (357.87,137.7) .. (348.81,140.19) ;
%Curve Lines [id:da7784383046544132] 
\draw    (348.81,140.19) .. controls (368.73,137.7) and (363.3,155.08) .. (380.51,159.54) ;
%Curve Lines [id:da2941524918555234] 
\draw    (383.22,131.75) .. controls (375.07,131.75) and (375.98,144.65) .. (368.28,145.65) ;
%Curve Lines [id:da22851678288586919] 
\draw    (363.3,149.12) .. controls (355.15,151.6) and (355.6,161.53) .. (346.55,164.01) ;
%Curve Lines [id:da4687027960050075] 
\draw    (346.55,164.01) .. controls (365.11,162.02) and (367.83,183.37) .. (379.6,183.86) ;
%Curve Lines [id:da7655733295255293] 
\draw    (380.51,159.54) .. controls (372.36,159.54) and (374.62,169.47) .. (366.92,170.46) ;
%Curve Lines [id:da6150839521834217] 
\draw    (361.94,175.92) .. controls (353.79,178.4) and (357.87,185.35) .. (348.81,187.83) ;
%Curve Lines [id:da3919765031060105] 
\draw    (348.81,187.83) .. controls (374.62,194.78) and (363.75,211.66) .. (375.98,212.65) ;
%Curve Lines [id:da2191275065824183] 
\draw    (379.6,183.86) .. controls (371.45,183.86) and (373.72,193.79) .. (366.02,194.78) ;
%Curve Lines [id:da9092037326219233] 
\draw    (360.58,198.75) .. controls (352.43,199.75) and (356.06,216.62) .. (347.9,216.62) ;
%Curve Lines [id:da7925490263100452] 
\draw    (375.98,212.65) .. controls (368.73,213.15) and (370.09,220.59) .. (366.02,225.55) ;
%Curve Lines [id:da2679142238120753] 
\draw    (497.92,175.43) .. controls (513.09,177.81) and (512.2,125.52) .. (532.89,126.72) .. controls (553.57,127.93) and (575.63,225.13) .. (596.7,191.72) ;
%Curve Lines [id:da1958852691308739] 
\draw    (679.14,201.63) .. controls (667.87,199.4) and (664.75,251.44) .. (644.07,250.08) .. controls (623.39,248.72) and (598.39,150) .. (580.43,184.58) ;
%Curve Lines [id:da5583106243844118] 
\draw    (630.59,175.7) .. controls (661.15,111.74) and (658.02,191.57) .. (679.14,201.63) ;
%Curve Lines [id:da9050717779799589] 
\draw    (557.04,144.86) .. controls (613.07,62.73) and (610.97,253.12) .. (661.49,179.31) ;
%Curve Lines [id:da6990477869608639] 
\draw    (611.02,143.64) .. controls (641.58,79.68) and (662.1,154.13) .. (682.94,152.34) ;
%Curve Lines [id:da9038083190568138] 
\draw    (682.94,152.34) .. controls (678.15,152.55) and (674.29,151.07) .. (667.19,168.28) ;
%Straight Lines [id:da36340701982251866] 
\draw    (623.91,191.45) -- (617.17,205.67) ;
%Straight Lines [id:da13632554965741028] 
\draw    (605.69,156.54) -- (599.93,175.77) ;
%Curve Lines [id:da9705436861181043] 
\draw    (497.92,175.43) .. controls (522.67,179.97) and (495.44,311.02) .. (558.02,227.12) ;
%Curve Lines [id:da8055446244179008] 
\draw    (549.12,154.65) .. controls (542.75,161.01) and (540.3,175.25) .. (532.19,175.74) ;
%Curve Lines [id:da8244888658446852] 
\draw    (532.19,175.74) .. controls (547.46,171.74) and (557.82,280.92) .. (611.39,220.33) ;
%Straight Lines [id:da6900907370733769] 
\draw    (574.9,194.76) -- (562.49,219.56) ;
\draw   (298.2,134.4) .. controls (301.41,132.27) and (303.82,129.77) .. (305.46,126.88) .. controls (304.61,130.15) and (304.54,133.79) .. (305.26,137.82) ;
\draw   (435.86,215.79) .. controls (439.06,213.67) and (441.48,211.17) .. (443.12,208.28) .. controls (442.27,211.55) and (442.2,215.19) .. (442.92,219.22) ;
\draw   (508.89,150.17) .. controls (512.37,147.81) and (515.11,144.78) .. (517.13,141.07) .. controls (515.84,145.45) and (515.27,150.5) .. (515.44,156.22) ;
\draw   (549.05,192.35) .. controls (549.04,198.09) and (549.73,203.05) .. (551.14,207.22) .. controls (549.03,203.83) and (546.21,201.21) .. (542.68,199.37) ;

% Text Node
\draw (330.01,264.41) node [anchor=north west][inner sep=0.75pt]   [align=left] {$\displaystyle x_{1}$};
% Text Node
\draw (335.44,145.78) node [anchor=north west][inner sep=0.75pt]   [align=left] {$\displaystyle x_{2}$};
% Text Node
\draw (333.63,198.39) node [anchor=north west][inner sep=0.75pt]   [align=left] {$\displaystyle x_{3}$};
% Text Node
\draw (339.52,117.99) node [anchor=north west][inner sep=0.75pt]   [align=left] {$\displaystyle x_{4}$};
% Text Node
\draw (335.9,170.1) node [anchor=north west][inner sep=0.75pt]   [align=left] {$\displaystyle x_{5}$};
% Text Node
\draw (525.03,255.68) node [anchor=north west][inner sep=0.75pt]  [rotate=-2.41] [align=left] {$\displaystyle y_{1}$};
% Text Node
\draw (585.47,152.18) node [anchor=north west][inner sep=0.75pt]  [rotate=-2.41] [align=left] {$\displaystyle y_{2}$};
% Text Node
\draw (620.15,104.18) node [anchor=north west][inner sep=0.75pt]  [rotate=-2.41] [align=left] {$\displaystyle y_{3}$};
% Text Node
\draw (651.59,190.53) node [anchor=north west][inner sep=0.75pt]  [rotate=-2.41] [align=left] {$\displaystyle y_{4}$};
% Text Node
\draw (525.84,150.98) node [anchor=north west][inner sep=0.75pt]  [rotate=-2.41] [align=left] {$\displaystyle y_{5}$};
% Text Node
\draw (621.87,192.63) node [anchor=north west][inner sep=0.75pt]  [rotate=-2.41] [align=left] {$\displaystyle y_{6}$};
% Text Node
\draw (628.28,144.11) node [anchor=north west][inner sep=0.75pt]  [rotate=-2.41] [align=left] {$\displaystyle y_{7}$};
% Text Node
\draw (569.83,199.8) node [anchor=north west][inner sep=0.75pt]  [rotate=-2.41] [align=left] {$\displaystyle y_{8}$};
\end{tikzpicture}
\caption{Two Legendrian knots with topological knot type $5_1$}
\label{fig6}
\end{figure}
\end{example}

\begin{example}\label{exmaple2}
Two Chekanov-Eliashberg knots are shown in Figure \ref{fig7}, say $K_3$ (left) and $K_4$ (right). It is easy to find that $tb(K_3)=tb(K_4)=1$ and $rot(K_3)=rot(K_3)=0$. And the fundamental GL-rack $GLR_{K_3}$ is generated by $\{ x_1,x_2,x_3,x_4,x_5,x_6 \}$ with relations:
\[
\begin{cases}
    u^2d(x_1) \ast x_4 = x_2\\
    x_2 \ast x_1 = x_3\\
    d^2(x_3) \ast x_6 = x_4\\
    u^2d(x_4) \ast x_1 = x_5\\
    d(x_5) \ast x_4 = x_6\\
    u(x_6) \ast x_3 = x_1
\end{cases}
\]
and the fundamental GL-rack $GLR_{K_4}$ is generated by $\{ y_1,y_2,y_3,y_4,y_5,y_6 \}$ with relations:
\[
\begin{cases}
    u^2d(y_1) \ast y_4 = y_2\\
    y_2 \ast y_1 = y_3\\
    d(y_3) \ast y_6 = y_4\\
    u^2(y_4) \ast y_1 = y_5\\
    d(y_5) \ast y_4 = y_6\\
    ud^2(y_6) \ast y_3 = y_1
\end{cases}
\]
Let $f$ be a map from $GLR_{K_3}$ to $GLR_{K_4}$ defined as below:
\begin{center}
$f(x_i)=y_i (i=1, 2, 3), f(x_4)=d(y_4), f(x_5)=d^2(y_5), f(x_6)=d^2(x_6).$
\end{center}
One can show that $f$ is an isomorphism and hence these two Legendrian knots have the isomorphic fundamental GL-racks. However, it is well known that these two knots are non Legendrian isotopic since they have distinct linearized contact homologies \cite{Chekanov}.

\tikzset{every picture/.style={line width=0.75pt}} %set default line width to 0.75pt        
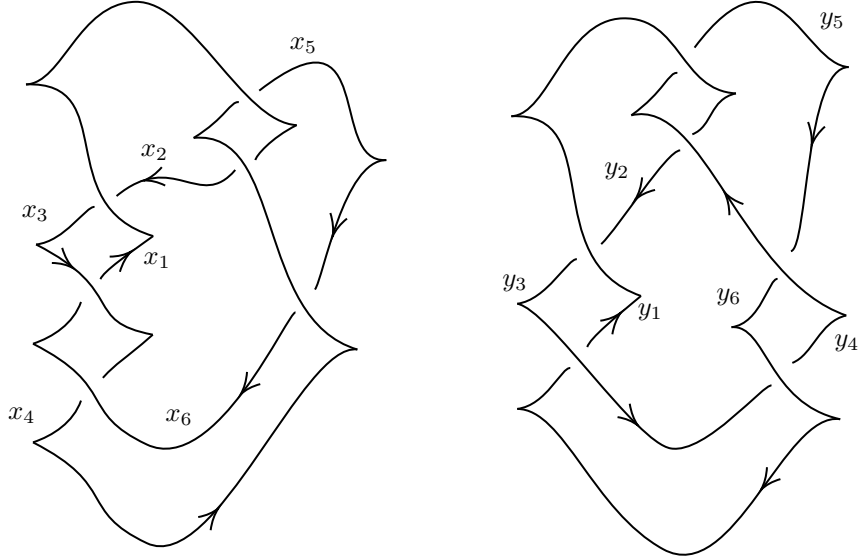
\begin{figure}[h]
\begin{tikzpicture}[x=0.75pt,y=0.75pt,yscale=-1,xscale=1]
%uncomment if require: \path (0,420); %set diagram left start at 0, and has height of 420

%Curve Lines [id:da16279447638565903] 
\draw    (314.74,118.75) .. controls (333.14,118.13) and (340.85,77.82) .. (369.35,77.2) .. controls (397.84,76.58) and (422.57,137.98) .. (449.88,139.22) ;
%Curve Lines [id:da3856468763745966] 
\draw    (314.74,118.75) .. controls (362.22,116.89) and (326.61,179.52) .. (378.25,195.03) ;
%Curve Lines [id:da5241292478610182] 
\draw    (378.25,195.03) .. controls (366.18,203.71) and (353.12,213.01) .. (351.93,216.73) ;
%Curve Lines [id:da45104153076603504] 
\draw    (319.78,199.26) .. controls (365.57,222.83) and (342.44,237.81) .. (378.05,244.63) ;
%Curve Lines [id:da29056844079044786] 
\draw    (378.05,244.63) .. controls (362.62,258.28) and (360.53,258.14) .. (353.12,266.34) ;
%Shape: Boxed Bezier Curve [id:dp8457454301371643] 
\draw    (318.24,249.18) .. controls (328.28,245.83) and (338.41,239.03) .. (342.21,228.17) ;
%Curve Lines [id:da4032505703422662] 
\draw    (318.24,249.18) .. controls (360.77,271.07) and (340.55,283.48) .. (375.06,299.87) .. controls (409.58,316.26) and (445.94,232.16) .. (449.29,233.47) ;
%Shape: Boxed Bezier Curve [id:dp6855342466276099] 
\draw    (318.24,298.79) .. controls (328.28,295.44) and (338.41,288.64) .. (342.21,277.78) ;
%Curve Lines [id:da4808957219936377] 
\draw    (318.24,298.79) .. controls (360.77,320.68) and (340.55,333.09) .. (375.06,349.48) .. controls (409.58,365.87) and (452.85,248.36) .. (480.15,252.08) ;
%Curve Lines [id:da08652672276753692] 
\draw    (398.43,145.65) .. controls (444.14,141.31) and (428.51,236.57) .. (480.15,252.08) ;
%Curve Lines [id:da5619871092278365] 
\draw    (449.88,139.22) .. controls (437.02,145.65) and (429.3,154.95) .. (429.3,156.81) ;
%Curve Lines [id:da8892921812856666] 
\draw    (359.95,174.9) .. controls (388.34,151.23) and (401.4,182.86) .. (419.21,161.77) ;
%Curve Lines [id:da1364122575357153] 
\draw    (398.43,145.65) .. controls (414.46,135.73) and (416.83,127.05) .. (420.99,127.67) ;
%Curve Lines [id:da5469800660195646] 
\draw    (431.18,121.88) .. controls (490.54,75.37) and (464.92,158.67) .. (494.6,156.81) ;
%Curve Lines [id:da6182763291985541] 
\draw    (458.98,221.92) .. controls (463.14,216.34) and (473.23,156.81) .. (494.6,156.81) ;
%Curve Lines [id:da9765985832073131] 
\draw    (557.01,134.75) .. controls (575.76,134.02) and (583.63,86.25) .. (612.67,85.51) .. controls (641.71,84.78) and (641.31,121.8) .. (669.14,123.27) ;
%Curve Lines [id:da7384989067590627] 
\draw    (557.01,134.75) .. controls (605.41,132.55) and (569.11,206.77) .. (621.74,225.15) ;
%Curve Lines [id:da9644980066903629] 
\draw    (621.74,225.15) .. controls (609.44,235.44) and (596.13,246.46) .. (594.92,250.87) ;
%Curve Lines [id:da5655753441073635] 
\draw    (666.92,240.86) .. controls (684.47,239.39) and (684.87,279.07) .. (721.17,287.16) ;
%Curve Lines [id:da3661481617609048] 
\draw    (559.97,229.06) .. controls (575.96,232.04) and (619.12,298.91) .. (636.67,301.85) .. controls (654.22,304.79) and (683.48,268.7) .. (686.89,270.25) ;
%Curve Lines [id:da7424723300489089] 
\draw    (559.97,281.97) .. controls (583.22,282.01) and (597.66,333.14) .. (632.84,352.56) .. controls (668.02,371.99) and (693.34,282.75) .. (721.17,287.16) ;
%Curve Lines [id:da125465891546361] 
\draw    (616.9,134.29) .. controls (653.41,125.48) and (671.56,216.6) .. (724.2,234.98) ;
%Curve Lines [id:da9288880172097653] 
\draw    (697.27,258.82) .. controls (707.51,254.85) and (710.89,237.92) .. (724.2,234.98) ;
%Curve Lines [id:da6483140791035494] 
\draw    (601.88,199.09) .. controls (610.89,190.36) and (634.71,151.83) .. (641.31,151.93) ;
%Curve Lines [id:da10910181584990164] 
\draw    (616.9,134.29) .. controls (633.24,122.54) and (635.66,112.25) .. (639.9,112.98) ;
%Curve Lines [id:da6404372592989834] 
\draw    (648.47,100.24) .. controls (690.32,43.9) and (702.42,108.57) .. (725.41,110.04) ;
%Curve Lines [id:da16872484384745334] 
\draw    (696.57,202.64) .. controls (705.04,201.17) and (703.63,110.04) .. (725.41,110.04) ;
%Curve Lines [id:da9849331056042976] 
\draw    (319.78,199.26) .. controls (334.13,197.74) and (344.81,179.76) .. (348.97,180.38) ;
%Curve Lines [id:da6811530075231403] 
\draw    (559.97,229.06) .. controls (574.59,227.26) and (585.48,205.95) .. (589.72,206.69) ;
%Curve Lines [id:da4551926890796374] 
\draw    (559.97,281.97) .. controls (574.59,280.18) and (582.22,260.7) .. (586.45,261.43) ;
%Curve Lines [id:da0207278460397895] 
\draw    (666.92,240.86) .. controls (681.54,239.06) and (685.68,215.87) .. (689.91,216.6) ;
%Curve Lines [id:da7898618740741408] 
\draw    (647.36,143.85) .. controls (657.6,139.88) and (655.83,126.21) .. (669.14,123.27) ;
\draw   (356.52,205.36) .. controls (360.64,205.35) and (364.2,204.5) .. (367.19,202.8) .. controls (364.77,205.35) and (362.91,208.74) .. (361.62,213) ;
\draw   (384.24,169.62) .. controls (380.49,167.75) and (376.93,166.92) .. (373.56,167.1) .. controls (376.74,165.88) and (379.74,163.63) .. (382.55,160.37) ;
\draw   (332.82,200.96) .. controls (334.07,205.22) and (335.91,208.64) .. (338.31,211.21) .. controls (335.34,209.48) and (331.78,208.6) .. (327.66,208.56) ;
\draw   (400.71,336.37) .. controls (404.78,335.62) and (408.15,334.13) .. (410.84,331.92) .. controls (408.84,334.87) and (407.53,338.56) .. (406.92,342.98) ;
\draw   (474.98,186.24) .. controls (471.67,188.92) and (469.29,191.9) .. (467.82,195.2) .. controls (468.36,191.59) and (467.98,187.66) .. (466.68,183.41) ;
\draw   (431.69,269.49) .. controls (427.69,270.54) and (424.41,272.27) .. (421.87,274.69) .. controls (423.67,271.59) and (424.75,267.82) .. (425.08,263.36) ;
\draw   (601.43,237.62) .. controls (605.52,237.01) and (608.93,235.64) .. (611.68,233.52) .. controls (609.6,236.4) and (608.19,240.04) .. (607.44,244.44) ;
\draw   (628.76,171.82) .. controls (624.73,172.76) and (621.42,174.4) .. (618.82,176.74) .. controls (620.7,173.7) and (621.86,169.96) .. (622.3,165.51) ;
\draw   (616.12,279.36) .. controls (616.82,283.77) and (618.18,287.43) .. (620.23,290.34) .. controls (617.5,288.18) and (614.11,286.75) .. (610.03,286.09) ;
\draw   (665.86,184.54) .. controls (665.82,180.07) and (664.99,176.21) .. (663.39,172.99) .. controls (665.77,175.59) and (668.92,177.57) .. (672.85,178.93) ;
\draw   (713.56,139.97) .. controls (710.4,142.83) and (708.18,145.96) .. (706.89,149.35) .. controls (707.24,145.71) and (706.66,141.8) .. (705.13,137.65) ;
\draw   (691.33,317.28) .. controls (687.35,318.42) and (684.1,320.22) .. (681.6,322.68) .. controls (683.36,319.56) and (684.36,315.75) .. (684.62,311.29) ;

% Text Node
\draw (371.87,201.7) node [anchor=north west][inner sep=0.75pt]   [align=left] {$\displaystyle x_{1}$};
% Text Node
\draw (370.64,147.47) node [anchor=north west][inner sep=0.75pt]   [align=left] {$\displaystyle x_{2}$};
% Text Node
\draw (311.1,178.42) node [anchor=north west][inner sep=0.75pt]   [align=left] {$\displaystyle x_{3}$};
% Text Node
\draw (304.34,279.6) node [anchor=north west][inner sep=0.75pt]   [align=left] {$\displaystyle x_{4}$};
% Text Node
\draw (444.92,93.55) node [anchor=north west][inner sep=0.75pt]   [align=left] {$\displaystyle x_{5}$};
% Text Node
\draw (382.92,280.6) node [anchor=north west][inner sep=0.75pt]   [align=left] {$\displaystyle x_{6}$};
% Text Node
\draw (618.64,227.66) node [anchor=north west][inner sep=0.75pt]   [align=left] {$\displaystyle y_{1}$};
% Text Node
\draw (602.06,156.45) node [anchor=north west][inner sep=0.75pt]   [align=left] {$\displaystyle y_{2}$};
% Text Node
\draw (551.11,213.05) node [anchor=north west][inner sep=0.75pt]   [align=left] {$\displaystyle y_{3}$};
% Text Node
\draw (716.86,242.97) node [anchor=north west][inner sep=0.75pt]   [align=left] {$\displaystyle y_{4}$};
% Text Node
\draw (709.49,80.9) node [anchor=north west][inner sep=0.75pt]   [align=left] {$\displaystyle y_{5}$};
% Text Node
\draw (657.93,220.03) node [anchor=north west][inner sep=0.75pt]   [align=left] {$\displaystyle y_{6}$};
\end{tikzpicture}
\caption{Chekanov-Eliashberg knots}
\label{fig7}
\end{figure}
\end{example}

\begin{remark}
A natural idea to promote GL-rack is to distinguish between the left up (dwon) cusp and the right up (down) cusp. Following this idea, recently Kimura introduced the notion of fundamental 4-Legendrian rack. However, the fundamental 4-Legendrian racks of the two Legendrian knots in Example \ref{exmaple2} are still isomorphic.
\end{remark}

\section{The proof of the main theorem}\label{section3}
For a Legendrian knot, we can obtain its fundamental GL rack as shown in Figure \ref{fig8}. At the two crossings, we have two relations $r_{i-1}:u^{p_{i-1}} d^{q_{i-1}} (x_{i-1}) \ast^{\epsilon_{i-1}} x_{k_{i-1}}=x_i$ and $r_i: u^{p_{i}} d^{q_{i}} (x_i) \ast^{\epsilon_i} x_{k_{i}}=x_{i+1}$, where $\epsilon_{i-1}$ and $\epsilon_i$ depend on the sign of the two crossings.

\tikzset{every picture/.style={line width=0.75pt}} %set default line width to 0.75pt        
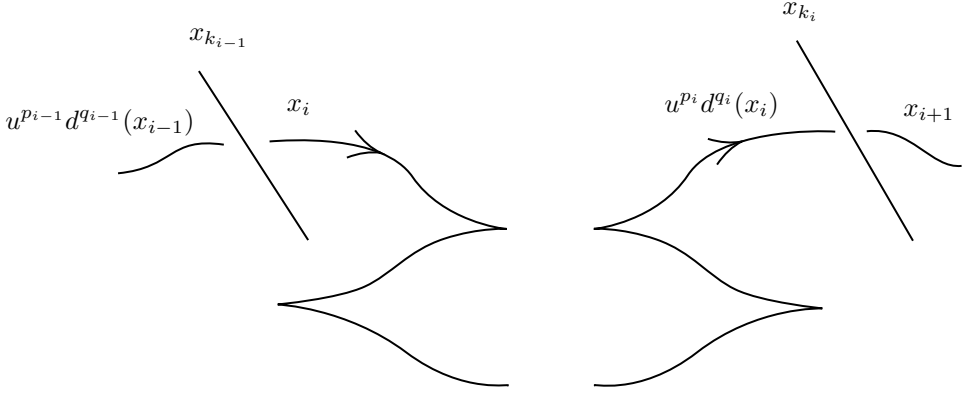
\begin{figure}
\begin{tikzpicture}[x=0.75pt,y=0.75pt,yscale=-1,xscale=1]
%uncomment if require: \path (0,419); %set diagram left start at 0, and has height of 419

%Shape: Boxed Bezier Curve [id:dp2208727293065591] 
\draw    (193.83,106.78) .. controls (202.83,106.12) and (250.23,102.66) .. (265.23,124.66) .. controls (280.23,146.66) and (309.9,150.77) .. (311.9,150.77) .. controls (313.9,150.77) and (298.83,149.78) .. (280.83,155.78) .. controls (262.83,161.78) and (254.79,173.58) .. (241.33,179.78) .. controls (227.88,185.99) and (197.9,188.77) .. (197.9,188.77) .. controls (197.9,188.77) and (232.27,190.4) .. (261.9,212.77) .. controls (291.53,235.13) and (315.8,228) .. (312.47,229.67) ;
%Shape: Boxed Bezier Curve [id:dp28438150779264915] 
\draw    (476.2,101.87) .. controls (467.2,101.14) and (417.76,100.87) .. (402.67,123.72) .. controls (387.58,146.57) and (357.9,150.74) .. (355.9,150.73) .. controls (353.9,150.73) and (368.97,149.76) .. (386.95,156.08) .. controls (404.92,162.39) and (412.92,174.71) .. (426.35,181.22) .. controls (439.78,187.74) and (469.75,190.75) .. (469.75,190.75) .. controls (469.75,190.75) and (436.99,189.66) .. (407.27,212.83) .. controls (377.54,236.01) and (353.14,227.88) .. (356.47,229.63) ;
\draw   (236.54,101.21) .. controls (240.35,106.5) and (244.63,110.28) .. (249.37,112.52) .. controls (244.17,111.8) and (238.53,112.59) .. (232.43,114.92) ;
\draw   (412.63,105.47) .. controls (418.78,107.34) and (424.43,107.78) .. (429.57,106.79) .. controls (424.94,109.23) and (420.81,113.11) .. (417.2,118.43) ;
%Shape: Free Drawing [id:dp8636564366630473] 
\draw  [line width=3] [line join = round][line cap = round] (321.4,229.07) .. controls (321.4,229.07) and (321.4,229.07) .. (321.4,229.07) ;
%Shape: Free Drawing [id:dp3108355446528953] 
\draw  [line width=3] [line join = round][line cap = round] (332.8,229.1) .. controls (332.8,229.1) and (332.8,229.1) .. (332.8,229.1) ;
%Shape: Free Drawing [id:dp5880980229662807] 
\draw  [line width=3] [line join = round][line cap = round] (343.8,229.1) .. controls (343.8,229.1) and (343.8,229.1) .. (343.8,229.1) ;
%Straight Lines [id:da22999706228911043] 
\draw    (158.56,71.39) -- (213.11,156.39) ;
%Straight Lines [id:da7339434488246935] 
\draw    (457,56.06) -- (515,156.72) ;
%Curve Lines [id:da09324019875566791] 
\draw    (491.93,101.67) .. controls (514.27,99.25) and (524.07,120.67) .. (539.27,119.07) ;
%Curve Lines [id:da1739753084350698] 
\draw    (118.2,122.8) .. controls (146.6,120) and (143.27,104.4) .. (171,108.4) ;

% Text Node
\draw (60.6,88.8) node [anchor=north west][inner sep=0.75pt]   [align=left] {$\displaystyle u^{p_{i-1}} d^{q_{i-1}}(x_{i-1})$};
% Text Node
\draw (152.13,48.67) node [anchor=north west][inner sep=0.75pt]   [align=left] {$\displaystyle x_{k_{i-1}}$};
% Text Node
\draw (201,84.4) node [anchor=north west][inner sep=0.75pt]   [align=left] {$\displaystyle x_{i}$};
% Text Node
\draw (389.55,80.31) node [anchor=north west][inner sep=0.75pt]  [rotate=-0.15,xslant=0.01] [align=left] {$\displaystyle u^{p_{i}} d^{q_{i}}( x_{i})$};
% Text Node
\draw (448.33,35.53) node [anchor=north west][inner sep=0.75pt]   [align=left] {$\displaystyle x_{k_{i}}$};
% Text Node
\draw (508.73,86.8) node [anchor=north west][inner sep=0.75pt]   [align=left] {$\displaystyle x_{i+1}$};
\end{tikzpicture}
\caption{Reading fundamental GL-rack from front projection}
\label{fig8}
\end{figure}

Now we give the proof of Theorem \ref{theorem1}.
\begin{proof}
Consider a front projection of $K_1$ and $K_2$, respectively, say $D_1$ and $D_2$. We assume that the arc collection of $D_1$ is $\{x_1, x_2, \dots ,x_m\}$, the arc collection of $D_2$ is $\{y_1, y_2, \dots ,y_n\}$. Then their fundamental GL-racks are 
    \[ GLR_{K_1} = \langle x_1, x_2, \dots, x_m \mid r_1, r_2, \dots, r_m \rangle \] and 
    \[ GLR_{K_2} = \langle y_1, y_2, \dots, y_n \mid \widetilde{r}_1, \widetilde{r}_2, \cdots, \widetilde{r}_n \rangle \] respectively, where $r_i$ $(1\leq i\leq m)$ and $\widetilde{r}_j$ $(1\leq j\leq n)$ are the relations derived from crossings of $D_1$ and $D_2$. The relation $r_i$ has the representation  \(u^{p_i} d^{q_i} (x_i) \ast^{\epsilon_i} x_{k_i} =x_{i+1} \) for \(1 \leq i \leq m\) and $\widetilde{r}_j$ has the representation  \( u^{\widetilde{p_j}} d^{\widetilde{q_j}}(y_j) \ast^{\widetilde {\epsilon_j}} y_{k_j} = y_{j+1} \) for \(1 \leq j \leq n\), where $x_{m+1} = x_1$, $y_{n+1} = y_1$, $ p_i ,q_i, \widetilde{p_j}, \widetilde{q_j} \in \mathbb{Z}_{\geq0}$ and $ \epsilon _i, \widetilde{\epsilon_j} \in \{1,-1\}$. We denote $ \displaystyle \sum_i p_i,\sum_i q_i, \sum_i \epsilon_i $ by $p,q,\omega $, respectively, and $ \displaystyle \sum_j \widetilde{p_j},\sum_j \widetilde{q_j} \sum_j \widetilde{\epsilon_j} $ by $\widetilde{p}, \widetilde{q}, \widetilde{\omega} $, respectively. 
    
    Consider the permutation GL-rack $(\mathbb{Z}_k, \ast, u, d)$ with $a\ast b=\sigma(a)$ and $ud=\sigma^{-1}$, where $a, b\in\mathbb{Z}_k$ and $\sigma$ is the $k$-cycle. Let $\phi: GLR_{K_1}\to(\mathbb{Z}_k, \ast, u, d)$ be a homomorphism from the fundamental GL-rack to $(\mathbb{Z}_k, \ast, u, d)$. Since $GLR_{K_1}$ and $GLR_{K_2}$ are isomorphic, then $\phi$ induces a homomorphism $\psi: GLR_{K_2}\to(\mathbb{Z}_k, \ast, u, d)$ from the fundamental GL-racks of $K_2$ to the permutation GL-rack $(\mathbb{Z}_k, \ast, u, d)$. Conversely, if $\psi$ is defined, then it also induces a homomorphism $\phi: GLR_{K_1}\to(\mathbb{Z}_k, \ast, u, d)$.
    
    For $GLR_{K_1}$, we can see that \(x_1 =u^{p_n} d^{q_n} (x_n) \ast^{\epsilon_n} x_{k_n} =u^{p_n+p_{n-1}} d^{q_n+q_{n-1}} ((x_{n-1}) \ast^{\epsilon_{n-1}} x_{k_{n-1}}) \ast^{\epsilon_n} x_{k_n} = \dots = u^p d^q (((x_1) \ast^{\epsilon_1} x_{k_1}) \ast^{\epsilon_2} \dots ) \ast ^{\epsilon_n} x_{k_n} \), then \( \phi(x_1) = u^p d^q \sigma^{\omega} (\phi (x_1)) \). Note that here the equalities $u(x\ast y)=u(x)\ast y$ and $ud(x)=du(x)$ are essentially used. Similarly, one obtains \( \psi (y_1) = u^{\widetilde{p}} d^{\widetilde{q}} \sigma^{\widetilde{\omega}} (\psi (y_1)) \).

    Now we choose $u = \sigma^{-1}$ and $d = id_{\mathbb{Z}_k}$, which follows that \( \phi (x_1) = \sigma^{\omega - p} (\phi (x_1)) \) and \( \psi (y_1) = \sigma^{\widetilde{\omega} - \widetilde{p}} (\psi (y_1)) \). We claim that \( \left| \omega -p \right| = \left| \widetilde{\omega} - \widetilde{p} \right| \). If not, assume \( \left| \omega -p \right| < \left| \widetilde{\omega} - \widetilde{p} \right| \), then there are two possibilities \( \left| \omega -p \right| =0 \) or \( \left| \omega -p \right| > 0 \).
    \begin{enumerate}
    \item If \( \left| \omega -p \right|=0 \), we set \( k >\left| \widetilde{\omega} - \widetilde{p} \right| \) and $\phi(x_1)=a$ for some element $a\in\mathbb{Z}_k$. Then $\phi(x_i)$ $(2\leq i\leq m)$ is completely determined by $\phi(x_1)$ and $r_1, \ldots, r_{m-1}$. The key point here is $\phi(x_1)=\sigma^{\omega-p}(\phi(x_1))$ since $\omega-p=0$, therefore the last relation $r_m$ is also satisfied. Hence $\phi$ defines a homomorphism from $GLR_{K_1}$ to $(\mathbb{Z}_k, \ast, u, d)$. Now we assume \( \psi (y_1) =b \in \mathbb{Z}_k \), then \( \sigma^{\widetilde{\omega} - \widetilde{p}}(\psi (y_1)) = \sigma^{\widetilde{\omega} - \widetilde{p}} (b) \ne b = \psi (y_1) \), which contradicts with the fact $\psi(y_1)=\sigma^{\widetilde{\omega}-\widetilde{p}}(\psi(y_1))$.
    \item If \( \left| \omega -p \right| > 0 \), we set \(k =\left| \widetilde{\omega} - \widetilde{p} \right| \) and $\psi(y_1)=b\in\mathbb{Z}_k$. Similar as above, now $\psi(y_j)$ $(2\leq j\leq n)$ is completely determined by $\psi(y_1)$ and $\widetilde{r}_1, \ldots, \widetilde{r}_{n-1}$ and $\psi$ defines a homomorphism from $GLR_{K_2}$ to $(\mathbb{Z}_k, \ast, u, d)$. Thus $\psi$ induces a homomorphism $\phi: GLR_{K_1}\to(\mathbb{Z}_k, \ast, u, d)$. Assume \( \phi(x_1) = a \in \mathbb{Z}_k \), then \( \sigma^{\omega - p} (\phi (x_1)) \ne a \) since \( \left| \omega -p \right| < \left| \widetilde{\omega} - \widetilde{p} \right| = k \), which contradicts with the fact that $\phi(x_1)=\sigma^{\omega-p}\phi(x_1)$.
    \end{enumerate}
    
    Now we have shown that \( \left| \omega -p \right| = \left| \widetilde{\omega} - \widetilde{p} \right| \). Similarly, choosing $u =id_{\mathbb{Z}_k}$ and $d = \sigma^{-1}$, we can obtain \( \left| \omega -q \right| = \left| \widetilde{\omega} - \widetilde{q} \right| \), and choosing $u = \sigma^{-2}$ and $d = \sigma$, we can obtain \( \left| \omega -2p +q \right| = \left| \widetilde{\omega} - 2\widetilde{p} + \widetilde{q} \right| \). Based on these, there are only two possibilities:
    \begin{enumerate}
    \item \( \omega - p = \widetilde{\omega} - \widetilde{p} \) and \( \omega - q = \widetilde{\omega} - \widetilde{q} \).
    \item \( \omega - p + \widetilde{\omega} - \widetilde{p} =0 \) and \( \omega - q + \widetilde{\omega} - \widetilde{q} =0 \).
    \end{enumerate}
    
Recall that the Thurston-Bennequin number and rotation number can be represented by the writhe and the number of cusps. For case $(1)$, we have \( \omega - p + \omega - q = \widetilde{\omega} - \widetilde{p} + \widetilde{\omega} - \widetilde{q} \) and \( p - q = \widetilde{p} - \widetilde{q} \), then \( tb (K_1) = tb (K_2)\) and \( rot(K_1) = rot(K_2) \). For case $(2)$, we have \( 2\omega - (p + q) + 2\widetilde{\omega} - (\widetilde{p} +\widetilde{q}) =0 \) and \( (q - p) + (\widetilde{q} - \widetilde{p}) = 0 \), then \( tb(K_1) + tb(K_2) =0 \) and \( rot(K_1) + rot(K_2) = 0 \). The proof is finished.
\end{proof}

Given a knot type $\mathcal{K}$, for all Legendrian knots $K \in \mathcal{K}$, according to the Bennequin's inequality \( \frac{1}{2}(tb(K)+\left|rot(K)\right|)+1\leq g_s (\mathcal{K}) \), where \( g_s (\mathcal{K}) \) is the slice genus of \( \mathcal{K} \). We obtain that if $K_1, K_2\in\mathcal{K}$ have isomorphic fundamental GL-racks, if $\mathcal{K}$ is slice, then their Thurston-Bennequin numbers are both negative. Thus, they have the same Thurston-Bennequin number and rotation number.

\begin{corollary}
Let $K_1, K_2$ be two Legendrian knots with isomorphic fundamental GL-rack, if one of the knot types is a slice knot, then \( tb(K_1) = tb(K_2) \) and \( rot(K_1) = rot(K_2) \).
\end{corollary}
\begin{proof}
Recall that the knot quandle can be recovered from the generalized GL-rack by setting $u=d=id$. It is well know that if two topological knots $\mathcal{K}_1$ and $\mathcal{K}_2$ have isomorphic knot quandles, then either $\mathcal{K}_1=\mathcal{K}_2$ or $\mathcal{K}_1=rm(\mathcal{K}_2)$ \cite{Joyce-1982,Matveev-1984}. Here $rm(\mathcal{K}_2)$ denotes the mirror image of the knot obtained from $\mathcal{K}_2$ by reversing the orientation. Therefore, if $\mathcal{K}_1$ is slice then $\mathcal{K}_2$ must also be slice. The result follows from the discussion above.
\end{proof}

We end this paper with the following question.

\begin{question}
Can we find two Legendrian knots with the same classical invariants but distinct fundamental GL-racks?
\end{question}

\bibliographystyle{plain}
\bibliography{ref}

\end{document}